\documentclass[10pt,twoside]{article}

\usepackage{amsfonts,amsthm,amssymb,latexsym,amsmath,enumerate,hyperref,color}

\usepackage{epsfig}
\usepackage{epstopdf}

\newtheorem{theorem}{Theorem}

\newtheorem{lemma}[theorem]{Lemma}

\newtheorem{definition}{Definition}

\newtheorem{example}[theorem]{Example}

\makeatletter
\evensidemargin \oddsidemargin
\makeatother

\begin{document}
\thispagestyle{empty}
\setcounter{page}{1}

\noindent
{\footnotesize {\rm Submitted to\\[-1.00mm]
{\em Dynamics of Continuous, Discrete and Impulsive Systems}}\\[-1.00mm]
http:monotone.uwaterloo.ca/$\sim$journal} $~$ \\ [.3in]


\begin{center}
{\large\bf ON TWO NONLINEAR DIFFERENCE EQUATIONS}

\vskip.20in

Jerico B. Bacani$^{1}$\ and\ Julius Fergy T. Rabago$^{2}$ \\[2mm]
{\footnotesize
Department of Mathematics and Computer Science\\
College of Science\\ 
University of the Philippines Baguio\\ 
Baguio City 2600, Benguet, PHILIPPINES\\[5pt]

E-mail: $^1$jicderivative@yahoo.com, $^2$jfrabago@gmail.com\\
}
\end{center}

{\footnotesize
\noindent
{\bf Abstract.}  The behavior of solutions of the following nonlinear difference equations
\[
x_{n+1}=\displaystyle\frac{q}{p+x_n^{\nu}}
\quad 
\text{and}
\quad y_{n+1}=\displaystyle\frac{q}{-p+y_n^{\nu}},
\]
where $p, q \in\mathbb{R}^+$ and $\nu\in \mathbb{N}$ are studied. 
The solution form of these two equations when $\nu =1$ are expressed in terms of Horadam numbers.
Meanwhile, the behavior of their solutions are investigated for all integer $\nu > 0$ and several numerical examples are presented to illustrate the results exhibited.
The present work generalizes those seen in [{\it Adv. Differ. Equ.}, {\bf 2013}:174 (2013), 7 pages].  \\[3pt]
}

\vskip.2in


\section{Introduction}
An equation of the form 
\begin{equation}\label{form}
x_{n+1} = f(x_n, x_{n-1}, \ldots, x_{n-k}),\quad n=0,1,\ldots
\end{equation}
where $f$ is a continuous function which maps some set $I^{k+1}$ into $I$ is called a \emph{difference equation of order} $k+1.$ 
The set $I$ is usually a sub-interval of the set of real numbers $\mathbb{R}$, a union of its sub-intervals, and may even be a discrete subset of $\mathbb{R}$ such as the set of \emph{integers} $\mathbb{Z}.$ 
A solution of \eqref{form}, uniquely determined by a prescribed set of $(k+1)$ \emph{initial conditions}  $x_{-k},x_{-k+1},\ldots,x_0 \in I$, is a sequence $\{x_n\}_{n=-k}^{\infty}$ which satisfies equation \eqref{form} for all $n\geq0.$ 
If for some least value $m\geq -k$, an initial point $(x_{-k},x_{-k+1},\ldots,x_0) \in I^{k+1}$ generates a solution $\{x_n\}$ with undefined value $x_m$, then we call the set $S$ of all such points the \emph{singularity set}, also called the \lq{}\lq{}forbidden set\rq{}\rq{} in the literature \cite{grove, kulenovic}.  
On the other hand, A solution of equation \eqref{form} which is constant for all $n\geq-k$ is called an \emph{equilibium solution} of \eqref{form}. 
If $x_n=\bar{x}$ for all $n\geq-k$ is an equilibrium solution of \eqref{form}, then $\bar{x}$ is called an \emph{equilibrium point}, or simply an \emph{equilibrium}, of \eqref{form}.
Difference equations are of great importance not only in the field of pure mathematics but also in the study and development of applied sciences. 
They appear naturally as discrete analogues and as numerical solutions of differential and delay differential equations which model various diverse phenomena in biology, ecology, physiology, physics, engineering, economics, etc. 
Recently, there has been an increasing interest in the study of qualitative analysis of rational difference equations and systems of difference equations. 
In fact, many research articles have been published previously in various mathematical journals devoted entirely in the investigation of these types of equations. 
Interestingly, these types of equations appear to have very simple forms, but, as it was seen in many literature, the behavior of their solutions are quite difficult to understand. 
So, difference equations are usually tackled by investigating the global character, boundedness, attractivity, oscillations and periodicity of their solutions. 
In an earlier paper \cite{Tollu}, Tollu et al. studied the form and behavior of solutions of the following difference equations  
\[
x_ {n+1} = \frac{1}{1+x_n}
\quad \text{and} \quad  
y_{n+1} = \displaystyle\frac{1}{-1+y_n}.
\]
Interestingly, it was shown in \cite{Tollu} that the solution form of the above equations are expressible in terms of Fibonacci numbers.
We mention that these two equations are, in fact, two special cases of the following Riccati difference equation:
\[
x_n=\frac{a+bx_n}{c+dx_n},\quad  n=0,1,\ldots,
\]    
which has been investigated recently by some authors, see, for instance, Brand \cite{Brand} and Papaschinopoulos and Papadopoulos \cite{Papaschinopoulos}. 
It is worth mentioning that the solution form has been found completely by Brand in \cite{Brand}.

Motivated by these aforementioned works, we investigate the form and behavior of solutions of the two nonlinear difference equations
\begin{equation}
\label{problem1}
x_{n+1}=\frac{q}{p+x_n^{\nu}},\quad  n=0,1,\ldots,
\end{equation} 
and
\begin{equation}\label{problem2}
y_{n+1}=\frac{q}{-p+y_n^{\nu}},\quad  n=0,1,\ldots,
\end{equation} 
where $p, q \in\mathbb{R}^+$ and $\nu\in \mathbb{N}$.
Particularly, we derive the form of solutions of the above equations in terms of Horadam numbers when $\nu = 1$ and investigate the long term dynamics of their solutions.
We also give conditions on the stability and instability of the equilibrium points of the above equations in terms of the parameters $p, q$ and $\nu$.
Furthermore, we provide numerical examples in confirming the results presented in the paper.

The paper is structured as follows: in Section 2, we discuss the well-known generalization of Fibonacci numbers called the \emph{Horadam sequence} and present some of its properties which will be useful in our investigation. 
In Section 3, we present the solution form of equations \eqref{problem1} and \eqref{problem2} and investigate their behaviors in terms of the relations between the parameter $p$ and $q$ for $\nu=1$.
Each results exhibited are illustrated through numerical examples. 
In Section 4, we give some results on the behavior of solutions of the two equations in consideration for the case $\nu > 1$ 
and accompany them with several numerical illustrations.  
Finally, a short summary and a statement of future work is given in Section 5.

\section{The Horadam Sequence}

In 1965, Horadam \cite{horadam} offered a generalization of Fibonacci sequence, that is, he defined a second-order linear recurrence sequence $\{W_n(a,b;p,q)\},$ or simply $\{W_n\},$ as follows:
$W_0=a$, $W_1=b$ and $W_{n+1}=pW_n + qW_{n-1}$ for all $n \geq 2$, where $a, b, p$ and $q$ are arbitrary real numbers. 
The Binet\rq{}s formula for this recurrence sequence can easily be obtained and is given by 
$
W_n=(A\Phi_+^n-B\Phi_-^n)/(\Phi_+ -\Phi_-)
$
where $A=b-a\Phi_-$ and $B=b-a\Phi_+.$ Here, $\Phi_{\pm}$ is simply the roots of the quadratic equation $x^2=px+q$, 
i.e., $\Phi_{\pm}= (p\pm\sqrt{p^2+4q})/2$.
Obviously, $\Phi_+ + \Phi_-=p$, $\Phi_+-\Phi_-=\sqrt{p^2+4q}$ and $\Phi_+\Phi_-=-q$. 
The sequence $\{W_n(0,1;p,q)\}$ can also be extended into negative indices with the recurrence relation
$
W_{-n}=-pW_{-n+1} + qW_{-n+2}.
$
That is, $W_{-n}=(-1)^{n+1}W_n$. 
It is worth mentioning that, for some specific values of $a, b, p$ and $q$, we\rq{}ll recover some well-known sequences other than Fibonacci sequence such as: $W_n(0,1;2,1)=P_n,$ the $n^{th}$ Pell number, and
$W_n(0,1;1,2)=J_n,$ the $n^{th}$ Jacobsthal number.
We mention the following properties of Horadam numbers which will be useful to our investigation.

\begin{lemma}\label{lemma1}
Let $W_0=0$ and $W_1=1.$ Then, we have the following identities:
\begin{enumerate}
\item[\rm (i)] For $n>k+1, n\in\mathbb{N}$ and $k\in\mathbb{N}\cup\{0\}$, $W_{n}=W_{k+1}W_{n-k}+qW_kW_{n-(k+1)}$.
\item[\rm (ii)] For $n>0, \;\;\Phi_{\pm}^n=\Phi_{\pm} W_n+W_{n-1}.$ 
\item[\rm (iii)] {\bf Cassini\rq{}s Formula.} For $n>0, \;\;W_{n-1}W_{n+1}-W_n^2=-(-q)^{n-1}$.
\item[\rm (iv)] {\bf d\rq{}Ocagne\rq{}s Identity.} For all $n,r \in\mathbb{N},$ we have
\[W_{n+r}W_{n+1}-W_{n+r+1}W_n = (-1)^nq^nW_r.\]
\item[\rm (v)] {\bf Johnson\rq{}s identity.}  For any integers $k$, $l$, $m$, $n$ and $r$ such that $k+l=m+n$, 
\[
W_kW_l-W_mW_n=(-q)^r(W_{k-r}W_{l-r}-W_{m-r}W_{n-r}).
\]
\end{enumerate}
\end{lemma}
Identity (i) of Lemma \ref{lemma1} can easily be verifed using induction on $n$. 
The proofs of  (i), (ii) and (iii), and (iv) can be found in \cite{Rabago2} and \cite{Rabago1}, respectively.
In addition to the above lemmas, we have, for any integer $r$,
\begin{equation}\label{infinityratio}
\lim_{n\rightarrow \infty} \frac{W_{n+r}}{W_n}=\Phi_+^r.
\end{equation}
For related papers and recent developments on these numbers, we refer the readers to a survey paper of Larcombe \cite{larcombe} and others \cite{lbf}.
Throughout our discussion we assume $W_n,$ the $n^{th}$ Horadam number, to satisfy the recurrence equation $W_{n+1}=pW_n + qW_{n-1}$ with initial conditions $W_0=0$ and $W_1=1$, unless specified.

\section{The case $\nu=1$}

In this section we study the case when $\nu =1$. 
Considering the difference equations defined in \eqref{problem1} and \eqref{problem2} for $\nu =1$, it is easy to see that the equilibrium points are 
$\bar{x} = \Phi_{\pm}-p$ and $\bar{y} = \Phi_{\pm}$,
respectively.  

\begin{theorem}\label{Theorem1}
The solutions of equations \eqref{problem1} and \eqref{problem2} are as follows:
\begin{enumerate}
\item[\rm (i)] For $x_0 \in \mathbb{R}\setminus\left(\left\{\frac{1}{\Phi_+}, \frac{1}{\Phi_-}\right\} \cup \left\{-\frac{W_{m+1}}{W_m}\right\}_{m=1}^{\infty}\right)$, $x_n= \dfrac{qW_n+x_0qW_{n-1}}{W_{n+1}+x_0W_n}$ for all $n \in \mathbb{N},$
\item[\rm (ii)] For $y_0 \in \mathbb{R}\setminus\left(\left\{\Phi_+, \Phi_-\right\} \cup \left\{\frac{W_{m+1}}{W_m}\right\}_{m=1}^{\infty}\right)$, $y_n= \dfrac{qW_{-n}+y_0qW_{-(n-1)}}{W_{-(n+1)}+y_0W_{-n}}$ for all $n \in \mathbb{N}.$
\end{enumerate}
\end{theorem}
\begin{proof} 
We only prove (ii). 
The proof of (i) is similar so we omit it. We proceed by induction. 
For $k=1$, 
\[
\frac{qW_{-1}+y_0qW_0}{W_{-2}+y_0W_{-1}}= \frac{q\cdot1+y_0q\cdot0}{-p+y_0\cdot1} = \frac{q}{-p+y_0}.
\] 
Now suppose that it holds for some natural number $k>1.$ Then,
\begin{equation}\label{waysabkey}
 y_k= \frac{qW_{-k}+y_0qW_{-(k-1)}}{W_{-(k+1)}+y_0W_{-k}}.
\end{equation}
We have, by equations \eqref{problem2} and \eqref{waysabkey},
\begin{align*}
y_{k+1} 
= \frac{qW_{-(k+1)}+y_0qW_{-k}}{-pW_{-(k+1)}+qW_{-k}+(-pW_{-k}+qW_{-(k-1)})y_0}
=  \frac{qW_{-(k+1)}+y_0qW_{-k}}{W_{-(k+2)}+y_0W_{-(k+1)}},
\end{align*}
proving the theorem.
\end{proof}

For convenience, we use the following notations in the rest of our discussion:
\[X(p,q) := \mathbb{R}\setminus\left(\left\{\frac{1}{\Phi_+(p,q)}, \frac{1}{\Phi_-(p,q)}\right\} \cup \left\{-\frac{W_{m+1}(p,q)}{W_m(p,q)}\right\}_{m=1}^{\infty}\right)\]
and
\[Y(p,q) := \mathbb{R}\setminus\left(\left\{\Phi_+(p,q), \Phi_-(p,q)\right\} \cup \left\{\frac{W_{m+1}(p,q)}{W_m(p,q)}\right\}_{m=1}^{\infty}\right),\]
where $\Phi_{\pm}$ is as defined before and $W_m(p,q)$ denotes the $m^{th}$ Horadam number with initial condition $W_0(p,q)=0$ and $W_1(p,q)=1$.
We provide the following example as an illustration of the previous theorem.
\begin{example}
\label{example1}
Consider the following Riccati difference equations whose solutions are associated to Pell numbers:
\[
x_{n+1} = \frac{1}{2+x_n}
\quad\text{and}\quad
y_{n+1} = \frac{1}{-2+y_n}.
\]
Using the results of Theorem \ref{Theorem1} we readily find the following respective solution form of the above equations:
\begin{enumerate}
\item[\rm (i)] for $x_0 \in X(2,1)$, $x_n= \dfrac{P_n+x_0P_{n-1}}{P_{n+1}+x_0P_n},$ for all $n \in \mathbb{N}$,

\item[\rm (ii)] for $y_0 \in Y(2,1)$, $y_n= \dfrac{P_{-n}+y_0P_{-(n-1)}}{P_{-(n+1)}+y_0P_{-n}}$ for all $n \in \mathbb{N}$.
\end{enumerate}
Here $\sigma$ denotes the silver ratio, i.e., $\sigma=1+\sqrt{2}$ and $P_n$ is the $n^{th}$ Pell number.
\end{example}

\begin{theorem}\label{Theorem3}
Let $\{x_n\}_{n=0}^{\infty}$ and $\{y_n\}_{n=0}^{\infty}$ be the solutions of \eqref{problem1} and \eqref{problem2}, respectively and $x_0\in \mathbb{R}\setminus\left\{-\frac{W_{m+1}}{W_m}\right\}_{m=1}^{\infty}$. 
Then, $\{x_n\}_{n=0}^{\infty}=\{-y_n\}_{n=0}^{\infty}$ is true if and only if the initial condition $x_0=-y_0.$
\end{theorem}
\begin{proof}
The above result maybe trivial to prove but we proceed on proving for the sake of completeness. 
First, suppose that $\{x_n\}_{n=0}^{\infty}=\{-y_n\}_{n=0}^{\infty}$. 
Then, by equation \eqref{problem2}, we have
\[
\frac{qW_n+x_0qW_{n-1}}{W_{n+1}+x_0W_n} 
= -\frac{qW_{-n}+y_0qW_{-(n-1)}}{W_{-(n+1)}+y_0W_{-n}} 
= \frac{qW_n-y_0qW_{n-1}}{W_{n+1}-y_0W_n}.
\]
Hence,
$
\frac{W_n+x_0W_{n-1}}{W_{n+1}+x_0W_n} 
= \frac{W_n-y_0W_{n-1}}{W_{n+1}-y_0W_n}.
$ 
With Cassini\rq{}s identity, the latter equation implies
\begin{align*}
(W_{n-1}W_{n+1}-W_n^2)x_0 &= -(W_{n-1}W_{n+1}-W_n^2)y_0 \\
(-q)^{n-1}x_0&=-(-q)^{n-1}y_0\\
x_0&=-y_0.
\end{align*}
On the other hand, suppose $x_0=-y_0$. 
Using equation \eqref{problem1}, we have
\begin{align*}
x_n &= \frac{qW_n+x_0qW_{n-1}}{W_{n+1}+x_0W_n}
= \frac{(-1)^{n+1}qW_n-(-1)^{n+1}y_0qW_{n-1}}{(-1)^{n+1}W_{n+1}-(-1)^{n+1}y_0W_n}\\
&= \frac{qW_{-n}+y_0qW_{-(n-1)}}{-W_{-(n+1)}-y_0W_{-n}}
= \frac{qW_{-n}+y_0qW_{-(n-1)}}{-(W_{-(n+1)}+y_0W_{-n})}
= -y_n.
\end{align*}
This proves the theorem.
\end{proof}
We illustrate our previous result with the following example.

\begin{example}
Consider the following nonlinear difference equations:
\[
x_{n+1} = \dfrac{7}{2+x_n}
\quad \text{and} \quad
y_{n+1} = \dfrac{7}{-2+y_n}
\]
with initial condition $x_0=3=-(-3)=y_0$. 
By Theorem \ref{Theorem3}, we have $\{x_n\}_{n=0}^{\infty}=\{-y_n\}_{n=0}^{\infty}$.  
The long term dynamics of the two nonlinear equations with the given initial conditions are shown in Figure \ref{fig1}.
\begin{figure}[h!]
   \centering
    \scalebox{.4}{\includegraphics{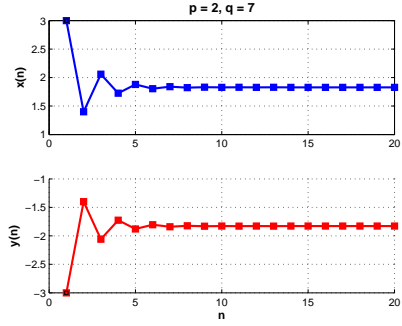}}
    \caption{Plot of $x_{n+1}=\frac{7}{2+x_n}$ and $y_{n+1}=\frac{7}{-2+y_n}$ with initial value $x_0 =-y_0=3$.}
    \label{fig1}
\end{figure}
\end{example}
\begin{theorem}
Given equations \eqref{problem1} and \eqref{problem2}, the following statements hold:
\begin{enumerate}
\item[\rm (i)] Let $x_0 = 1/\Phi_{\pm}$ then \eqref{problem1} has a fixed solution $x_n = q/\Phi_{\pm}.$
\item[\rm (ii)]  Let $y_0 = \Phi_{\pm}$ then \eqref{problem2} has a fixed solution $y_n = \Phi_{\pm}.$
\end{enumerate}
\end{theorem}
\begin{proof}
We only prove (i). The proof of (ii) can be done in a similar fashion.
So consider equation \eqref{problem1} and let $x_0=1/\Phi_+$. 
Then, using identity (ii) of Lemma \ref{lemma1}, we have
\[
x_n = \frac{qW_n+q\frac{W_{n-1}}{\Phi_{\pm}}}{W_{n+1}+\frac{W_n}{\Phi_{\pm}}}
=\frac{\Phi_{\pm}W_n+W_{n-1}}{\Phi_{\pm}W_{n+1}+W_n}q 
=\frac{\Phi_{\pm}^n}{\Phi_{\pm}^{n+1}}q
=\frac{q}{\Phi_{\pm}},
\] 
which is desired.
\end{proof}

\begin{theorem}
The following statements hold:

\begin{enumerate}
\item[\rm (i)] For $x_0 \in X(p,q)\cup\left\{1/\Phi_+(p,q)\right\}$, we have $\lim_{n\rightarrow \infty} x_n= -\Phi_-,$

\item[\rm (ii)] For $y_0 \in Y(p,q)\cup\left\{\Phi_-(p,q)\right\}$, we have $\lim_{n\rightarrow \infty} y_n= \Phi_-.$ 
\end{enumerate}
\end{theorem}
\begin{proof}
The proof of the above statements are similar. 
Hence, we only show (i). 
One may prove this by simply letting $n \rightarrow \infty $ in Theorem \ref{Theorem3} and use equation \eqref{infinityratio}. 
Instead, we provide another approach. Suppose $\lim_{n\rightarrow \infty} x_n = L$. 
Then, letting $n\rightarrow\infty$ on both sides of \eqref{problem1}, 
we get $L=q/(p+L)$ or equivalently, $L^2 + p L -q=0$.
Solving for $L$, we obtain the desired result.
\end{proof}

\begin{example}
\label{example2}
As an example, we consider the equations considered in Example \ref{example1} with initial conditions
$x_0 \in X(2,1) \cup \{1/\sigma\}$
and
$y_0 \in Y(2,1)\cup\{2-\sigma\}$,
respectively. Hence, $\lim_{n\rightarrow\infty} x_n = \sigma-2 = 0.4142$ (approx.) and $\lim_{n\rightarrow\infty} y_n=2-\sigma = -0.4142$ (approx.). 
The results are illustrated in Figure \ref{fig2} with initial conditions $x_0=2$ and $y_0=3$, respectively.
\begin{figure}[h!]
   \centering
    \scalebox{0.4}{\includegraphics{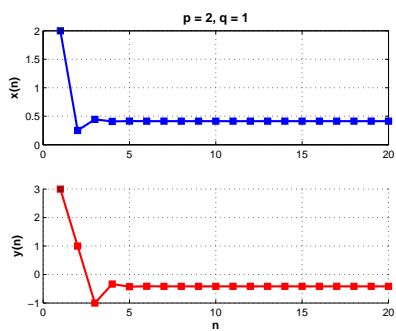}}
    \caption{Plot of $x_{n+1}=\frac{1}{2+x_n}$ and $y_{n+1}=\frac{1}{-2+y_n}$ with initial values $x_0=2$ and $y_0=3$, respectively.}
    \label{fig2}
\end{figure}
 \end{example} 
\begin{theorem}\label{theorem8}
Let $p> q-1$ and $\{x_n\}_{n=0}^{\infty}$ (resp. $\{y_n\}_{n=0}^{\infty}$) be the solution of \eqref{problem1} (resp. \eqref{problem2}). 
Then, $\prod_{i=0}^n x_i \rightarrow W_0$ (resp. $\prod_{i=0}^n y_i \rightarrow W_0$) as $n \rightarrow \infty$.
\end{theorem}

\begin{proof}
We only prove the result for \eqref{problem1}. 
The same inductive lines, however, can be followed inductively to obtain a similar result for equation \eqref{problem2}.
Now, the case when $x_0=W_0$ is obvious. 
If $x_0\neq W_0$ then, in view Theorem \ref{Theorem1}, we have
\[
x_1 = \frac{W_1 + x_0W_0}{W_2+x_0W_1}q, \quad
x_2 = \frac{W_2 + x_0W_1}{W_3+x_0W_2}q, \quad
\ldots,\quad
x_n  = \frac{W_n + x_0W_{n-1}}{W_{n+1}+x_0W_n}q.
\]
Hence, 
\[
\prod_{i=0}^nx_i=\frac{W_0+x_0W_{-1}}{W_{n+1}+x_0W_n}q^n
=\frac{q^nx_0}{W_{n+1}+x_0W_n}.
\]
Observe that $\prod_{i=0}^nx_i\rightarrow 0$ as $n\rightarrow \infty$ for $q\in(0,1]$. 
Now, if $q>1$ then we have
\begin{align*}
\lim_{n\rightarrow \infty} \prod_{i=0}^nx_i 
&= \lim_{n\rightarrow \infty}\frac{q^nx_0}{W_{n+1}+x_0W_n}
= \frac{x_0\lim_{n\rightarrow \infty} \frac{q^n}{W_n}}{\lim_{n\rightarrow \infty}\frac{W_{n+1}}{W_n}+x_0}\\
&=\frac{x_0\sqrt{p^2+4q}}{\Phi_++x_0} \left(\lim_{n\rightarrow \infty}\frac{\Phi_+^n}{q^n}-\lim_{n\rightarrow \infty}\frac{\Phi_-^n}{q^n}\right)^{-1}.
\end{align*}
Note that the inequality $p > q-1$ implies that $\Phi_+/q > 1$ and 
$
-1
<\Phi_-/q
<0$.
The right hand part of the latter inequality can be verified easily.
Meanwhile, to see $\Phi_-/q> -1$, we use the fact that $\sqrt{a+b} < \sqrt{a} +\sqrt{b}$ holds for all $a,b\in\mathbb{R}^+$.
Hence, $p+2q-\sqrt{p^2+4q}> p+2q - (\sqrt{p} + \sqrt{2q}) > 0$ which implies that $p-\sqrt{p^2+4q} > -2q$.
Upon dividing both sides of the latter inequality by $2q$, we get the desired result.
So $(\Phi_+/q)^{n+1} \rightarrow \infty$ and $(\Phi_-/q)^{n+1} \rightarrow 0$ as $n\rightarrow\infty$. 
Thus, $\lim_{n\rightarrow \infty} \prod_{i=0}^nx_i =W_0$, proving the theorem.
\end{proof}

Figure \ref{fig3} illustrates the product of solutions $\{x_n\}_{n=0}^{\infty}$ and $\{y_n\}_{n=0}^{\infty}$ of $x_{n+1}=1/(2+x_n)$ and $y_{n+1}=1/(-2+x_n)$, respectively which have been considered in Example \ref{example1} with initial condition $x_0=-y_0=2$.
\begin{figure}[h!]
   \centering
    \scalebox{.4}{\includegraphics{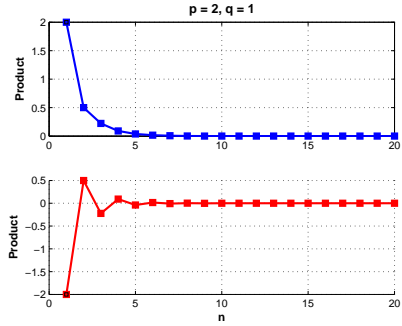}}
     \caption{The product of solutions of the two difference equations $x_{n+1}=\frac{1}{2+x_n}$ and $y_{n+1}=\frac{1}{-2+y_n}$ from $n=0$ to 20 are illustrated in the above figure (respectively, upper and lower plot).}
    \label{fig3}
\end{figure}

\begin{theorem}\label{theorem9}
Let $x_0\neq -\Phi_+$ (resp. $y_0 \neq -\Phi_+$) and $\{x_n\}_{n=0}^{\infty}$ (resp. $\{y_n\}_{n=0}^{\infty}$) be the solution of equation \eqref{problem1} (resp. equation \eqref{problem2}). 
Then, the following statements are true: 
\begin{enumerate}
\item[\rm (i)] If $p = q-1$, then we have 
\[\lim_{n\rightarrow \infty} \prod_{i=0}^nx_i 
=\dfrac{x_0\sqrt{p^2+4q}}{\Phi_++x_0} 
= \begin{cases} 
	\quad\lim_{n\rightarrow \infty} \prod_{i=0}^ny_i, & \text{when $n$ is even},\\[1em]
	-\lim_{n\rightarrow \infty} \prod_{i=0}^ny_i, & \text{when $n$ is odd},
\end{cases}.
\]
\item[\rm (ii)] If $p<q-1$, then the limits $\lim_{n\rightarrow \infty} \prod_{i=0}^nx_i$ and $\lim_{n\rightarrow \infty} \prod_{i=0}^ny_i$ diverges.
\end{enumerate}
\end{theorem}

\begin{proof}
Again, we only prove the result for equation \eqref{problem1} and omit the proof for the corresponding result for equation \eqref{problem2} since they are imilar.
First, we assume that $p=q-1$. Hence, $\Phi_+/q=1$ and $\Phi_-/q=-1/q$. 
Now, from the proof of Theorem \ref{theorem8}, we have seen that
\[
\lim_{n\rightarrow \infty} \prod_{i=0}^nx_i 
=\frac{x_0\sqrt{p^2+4q}}{\Phi_++x_0} \left[\lim_{n\rightarrow \infty}\left(\frac{\Phi_+}{q}\right)^n-\lim_{n\rightarrow \infty}\left(\frac{\Phi_-}{q}\right)^n\right]^{-1}.
\]
Thus, $\lim_{n\rightarrow \infty} \prod_{i=0}^nx_i =(x_0\sqrt{p^2+4q})/(\Phi_++x_0)$, proving (i). 
On the other hand, if $p<q-1$ then
\[
0<\frac{\Phi_+}{q} 
< \frac{(q-1)+\sqrt{(q-1)^2+4q}}{2q} < 1.
\]
Furthermore, $p<q-1$ implies
\[
-1
<
\frac{p}{q}-1
=\frac{2(p-q)}{2q}
<\frac{(p-q)-1}{2q}
<\frac{p+\sqrt{(q-1)^2+4q}}{2q}
<\frac{\Phi_-}{q} <0.
\]
So it follows that $\lim_{n\rightarrow\infty} (\Phi_+/q)^n=0$ and $\lim_{n\rightarrow\infty} (\Phi_-/q)^n=0.$ Therefore, the limit $\lim_{n\rightarrow \infty} \prod_{i=0}^nx_i$ diverges. 
This concludes statement (ii), completing the proof of the theorem.
\end{proof}

\begin{example}
Consider the nonlinear difference equation $x_{n+1} = 2/(1+x_n)$ whose solutions are associated to Jacobsthal numbers.
Recall that $J_n=W_n(0,1;1,2)$. 
Hence, $p=1=2-1=q-1$. 
Furthermore, $\Phi_+=2$ and $\Phi_-=-1$. 
It follows from Theorem \ref{theorem9} that $\lim_{n\rightarrow \infty} \prod_{i=0}^n x_i = 3x_0/(2+x_0)$,
i.e., for $x_0=9$, we have $\lim_{n\rightarrow\infty}\prod_{i=0}^n x_i = 27/11 = 2.454552264$ (approx.).
Also, in reference to Theorem \ref{theorem9}, we see that 
\[
\lim_{n\rightarrow\infty}\prod_{i=0}^n y_i = \begin{cases} \ \ \frac{27}{11}, & \text{when $n$ is even},\\[1em] -\frac{27}{11}, & \text{when $n$ is odd}, \end{cases}
\]
where $\{y_n\}_{n=1}^{\infty}$ is the solution of the difference equation $y_{n+1} = 2/(-1+y_n)$ with initial condition $y_0=-x_0=-9$, refer to Figure \ref{fig4} for the plots. 
\begin{figure}[h!]
   \centering
    \scalebox{0.4}{\includegraphics{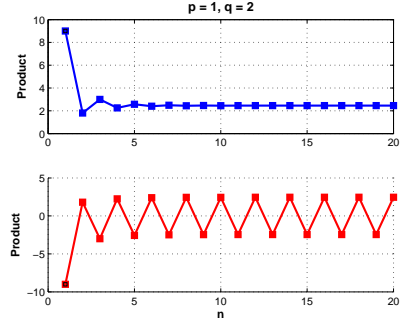}}
    \caption{The product of solutions of $x_{n+1}=\frac{2}{1+x_n}$ and $y_{n+1}=\frac{2}{-1+y_n}$ from $n=0$ to 20 are shown in the above figure (respectively, upper and lower plot).}
    \label{fig4}
\end{figure}
\end{example}

\begin{example}
As an example for Theorem \ref{theorem9}-(ii), we consider the two nonlinear difference equations $x_{n+1} = 2/(1/2+x_n)$ and $y_{n+1} = 2/(-1/2+y_n)$ with the same initial conditions as in the previous example.
The respective product of their solutions diverges as $n \rightarrow \infty$ and these are shown in Figure \ref{fig5}.
\begin{figure}[h!]
   \centering
    \scalebox{0.4}{\includegraphics{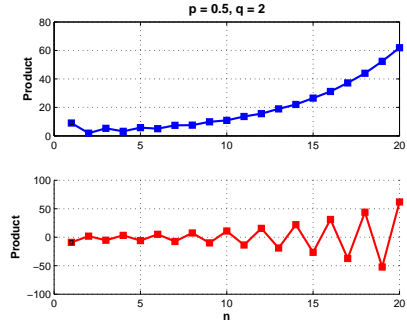}}
    \caption{The product of solutions of $x_{n+1}=\frac{2}{1/2+x_n}$ and $y_{n+1}=\frac{2}{-1/2+y_n}$ from $n=0$ to 20 are shown in the above figure (respectively, upper and lower plot).}
    \label{fig5}
\end{figure}
\end{example}

\begin{theorem}\label{theorem12}
Consider equation \eqref{problem1} with initial condition $x_0=qW_k/W_{k+1}$.
Then, for $n,k \in \mathbb{N}$ and $n>k+1$, we have
$W_n = q^{n-(k+1)}W_{k+1}/\prod_{i=1}^{n-(k+1)}x_i$. 
\end{theorem}
\begin{proof}
Consider the product $x_0\cdot\prod_{i=1}^{n-(k+1)} x_i = q^{n-(k+1)}x_0/(W_{n-k}+x_0W_{n-(k+1)})$.
Hence,
\begin{align*}
\prod_{i=1}^{n-(k+1)} x_i
&=\frac{q^{n-(k+1)}}{W_{n-k}+x_0W_{n-(k+1)}}
=\frac{q^{n-(k+1)}}{W_{n-k}+\frac{qW_k}{W_{k+1}}\cdot W_{n-(k+1)}}\\
&=\frac{q^{n-(k+1)}W_{k+1}}{W_{n-k}W_{k+1}+qW_k W_{n-(k+1)}}.
\end{align*}
By identity (i) of Lemma \ref{lemma1}, we obtain the desired result.
\end{proof}

We provide the following example for the previous theorem.

\begin{example} 
Consider, for instance, the nonlinear difference equation $x_{n+1}=1/(1+x_n)$ studied by Tollu et al. in \cite{Tollu}. 
If we let $n=15, k=1$ and $x_0=(1)(F_1/F_2)=1$, then we have 
$x_1\cdot x_2\cdot\cdots \cdot x_{13} = F_2/F_{15}=1/610$.
\end{example}

We also have the following theorems.

\begin{theorem}
Consider equation \eqref{problem1} with initial condition $x_0=-W_{n+r+1}/W_{n+r}$. 
Then, for all $n,r\in \mathbb{N}$, we have $(-1)^n\prod_{i=1}^n x_i = W_{n+r}/W_r$.
Furthermore, we have the limit $\lim_{r\rightarrow\infty}(-1)^n\prod_{i=1}^n x_i = \Phi_+^r$.
\end{theorem}

\begin{proof}
We follow the proof of Theorem \ref{theorem12}, that is, we consider the following product
\[
x_0\cdot\prod_{i=1}^n x_i=\frac{q^nx_0}{W_{n+1}+x_0W_n},
\]
with initial condition $x_0=-W_{n+r+1}/W_{n+r}$ where $n,r \in \mathbb{N}$.
Hence, using d\rq{}Ocagne\rq{}s identity, we have
\begin{eqnarray*}
\prod_{i=1}^n x_i=
\frac{q^nW_{n+r}}{W_{n+r}W_{n+1}-W_{n+r+1} W_n}
=\frac{W_{n+r}}{(-1)^nW_r}.
\end{eqnarray*}
Multiplying both sides by $(-1)^n$ and letting $r\rightarrow\infty$, we obtain the limit
$\lim_{r\rightarrow\infty}(-1)^n\prod_{i=1}^n x_i = \lim_{r\rightarrow\infty}W_{n+r}/W_r=\Phi_+^r$,
proving the theorem.
\end{proof}

The next theorem is our final result for this section. 

\begin{theorem}
Let $r\in\mathbb{N}$ and $x_0=-W_{r+1}/W_r$ be the initial condition of \eqref{problem1}, where $W_r$ is the $r^{th}$ Horadam number. Then, we have
$(-1)^{r+1}q^{r-n}\prod_{i=1}^n x_i = W_r/W_{n-r}$.  
\end{theorem}

\begin{proof}
Again, we consider the product $\prod_{i=1}^n x_i =q^n/(W_{n+1}+x_0W_n)$
with initial value $x_0=-W_{r+1}/W_r$, where $r\in \mathbb{N}$. 
Then, we have
\begin{align*}
\prod_{i=1}^n x_i &=\frac{q^n}{W_{n+1}-\frac{W_{r+1}}{W_r}W_n}
=\frac{q^nW_r}{W_rW_{n+1}-W_{r+1}W_n}\\
&= \frac{q^nW_r}{(-q)^r(W_{r-r}W_{(n+1)-r}-W_{(r+1)-r}W_{n-r})}
=\frac{q^nW_r}{(-1)^{r+1}q^rW_{n-r}}.
\end{align*}
Rearranging the latter equation, we get $(-1)^{r+1}q^{r-n}\prod_{i=1}^n x_i =  W_r/W_{n-r}$, which is desired.
\end{proof}

In the next section we tackle the case when $\nu>1$ in equations \eqref{problem1} and \eqref{problem2}. 

\section{The case $\nu >1$}

We first introduce some basic definitions and some theorems that we need in the sequel. 
Let $I$ be some interval of real numbers and let $f: I^{k+1} \rightarrow I$ be a continuously differentiable function. Then, for every set of initial conditions $x_{-k},x_{-k+1},\ldots,x_0 \in I$, the difference equation
\begin{equation}\label{function}
x_{n+1}=f(x_{n},x_{n-1},\ldots,x_{n-k}),\quad n=0,1,\ldots
\end{equation}
has a unique solution $\{x_n\}_{n=-k}^{\infty}$ (cf. \cite{kocic}).

\begin{definition}[Stability]
\begin{enumerate}
\item[(i)] The equilibrium point $\bar{x}$ of \eqref{function} is locally stable if, for every $\epsilon>0$, there exists $\delta$ such that for all $x_{-k}$, $x_{-k+1}$, $\ldots,x_0 \in I$ with 
$\sum_{i=-k}^0|x_{-i}-\bar{x}| < \delta$ we have $|x_n-\bar{x}|<\epsilon$ for all $n\geq -k$.
\item[(ii)] The equilibrium point $\bar{x}$ of $\eqref{function}$ is locally asymptotically stable if $\bar{x}$ is locally stable solution of \eqref{function} and there exists $\gamma > 0$, 
such that for all $x_{-k}$, $x_{-k+1}$, $\ldots,x_0 \in I$ with 
$ \sum_{i=-k}^0|x_{-i}-\bar{x}|  <\delta$ we have $\lim_{n\rightarrow\infty} x_n =\bar{x}$.
\item[(iii)] The equilibrium point $\bar{x}$ of $\eqref{function}$ is a global attractor if, for all $x_{-k}$, $x_{-k+1}$, $\ldots,x_0 \in I$, 
we have  $\lim_{n\rightarrow\infty} x_n =\bar{x}$.
\item[(iv)] The equilibrium point $\bar{x}$ of $\eqref{function}$ is a globally asymptotatically stable if $\bar{x}$ is locally stable, and $\bar{x}$ is also a global attractor of \eqref{function}.
\item[(v)] The equilibrium point $\bar{x}$ of $\eqref{function}$ is unstable if $\bar{x}$ is not locally stable.
\end{enumerate}

The linearized equation of \eqref{function} about the equilibrium $\bar{x}$ is the linear difference equation
\[
y_{n+1} = \sum_{i=0}^k \frac{\partial f(\bar{x},\bar{x},\ldots,\bar{x})}{\partial x_{n-i}}y_{n-i}.
\]
\end{definition}

\begin{theorem}[\cite{kulenovic}]\label{theoremA} Assume that $p_i \in \mathbb{R}, i = 0,1,\ldots,k$ and $k \in \{0, 1, 2, \ldots\}$. 
Then, $\sum_{i=1}^k|p_i|<1$ is a sufficient condition for the asymptotic stability of the difference equation:
\[
x_{n+k}+p_1x_{n+k-1}+\ldots+p_kx_{n}=0,\quad n=0,1,\ldots.
\]
\end{theorem}

\begin{definition}[Periodicity]
A sequence $\{x_n\}_{n=-k}^{\infty}$ is said to be periodic with period $p$ if $x_{n+p}=x_n$ for all $n \geq -k$.
\end{definition}

\begin{definition}[\cite{grove}]
A solution $\{x_n\}_{n=-k}^{\infty}$ of \eqref{function} is called \emph{eventually periodic with period} $p$ if there exists an integer $N\geq-k$ such that  $\{x_n\}_{n=N}^{\infty}$ is periodic with period $p$; 
that is, $x_{n+p}=x_n$, for all $n \geq N$.
\end{definition}
Now, we are in the position to investigate the case when $\nu >1$.

\subsection{On equation $x_{n+1}=q/(p+x_n^{\nu})$}

We have the following theorems.

\begin{theorem}\label{boundedness1}
Every positive solution of \eqref{problem1} is bounded. 
\end{theorem}
\begin{proof}
Let $\{x_n\}_{n=0}^{\infty}$ be a solution to \eqref{problem1}. 
Then, $ x_{n+1} = q/(p+x_n^{\nu}) \leq q/p$. 
Hence, $x_n^{\nu} \leq (q/p)^{\nu}$ which implies that $x_{n+1} = q/(p+x_n^{\nu}) \geq q/(p+(q/p)^{\nu})$. 
Thus, $q/(p+(q/p)^{\nu}) \leq x_n \leq q/p$.
\end{proof}

\begin{theorem}\label{fixedpoints1}
Let $\bar{x}$ be an equilibrium point of equation \eqref{problem1}. Then, the following statements are true:
\begin{enumerate}
\item[\rm (i)] if $q<p+1$ then \eqref{problem1} has a unique positive equilibrium $\bar{x}$ and $\bar{x} \in (0,1),$
\item[\rm (ii)] if $q=p+1$ then $\bar{x}=1$ is a unique positive equilibrium of \eqref{problem1},
\item[\rm (iii)] if $q>p+1$ then \eqref{problem1} has a unique positive equilibrium $\bar{x}$ and $\bar{x} >1.$
\end{enumerate}
\end{theorem}
\begin{proof}
Let $\bar{x}$ be an equilibrium of \eqref{problem1} and consider the function $F(x) = x^{\nu+1}+px-q$.
We first show that \eqref{problem1} has a unique positive equilibrium for any $p,q \in \mathbb{R}^+$.
We have $F\rq{}(x)=(\nu+1)x^{\nu}+p$. 
Then, $F\rq{}(x)=0$ if and only if $x=\left(-p/(\nu+1)\right)^{1/\nu}$.
It follows that $F\rq{}(x)>0$ and $F(x)$ is increasing in $(0,\infty)$. 
Moreover, $F(0)=-q<0$ and $\lim_{x\rightarrow +\infty} F(x)= +\infty$.
Thus, for any $p,q\in\mathbb{R}^+$, \eqref{problem1} has unique equilibrium in $(0,\infty)$. 
Now,
\[
\bar{x} = \frac{q}{p+\bar{x}^{\nu}}
\quad\Longleftrightarrow\quad
\bar{x}\left(1+\frac{1}{p}\bar{x}^{\nu}\right)=\frac{q}{p}.
\]
Let $q<p+1$. Then, 
\[
\bar{x}\left(1+\frac{1}{p}\bar{x}^{\nu}\right)< 1+\frac{1}{p}.
\]
Suppose $\bar{x}\geq 1$. Then, 
\begin{eqnarray*}
	\bar{x}\geq 1 
		&\Longleftrightarrow& \bar{x}^{\nu} \geq 1, \quad\quad \text{for all $\nu >1$}\\
		&\Longleftrightarrow& \frac{1}{p}\bar{x}^{\nu} \geq \frac{1}{p}, \quad \text{for all $p \in \mathbb{R}^+$}\\
		&\Longleftrightarrow& 1+\frac{1}{p}\bar{x}^{\nu} \geq 1+\frac{1}{p}\\
		&\Longleftrightarrow& \bar{x}\left( 1+\frac{1}{p}\bar{x}^{\nu} \right)\geq 1+\frac{1}{p}, \quad \text{for all $\bar{x}\geq 1$},
 \end{eqnarray*} 
 a contradiction. Thus, $\bar{x}<1$. 
If $q=p+1$, then $\bar{x}\left(p+\bar{x}^{\nu}\right)= p+1$. 
Obviously, $\bar{x}=1$.  
In fact, for $q=p+1,$ the polynomial $x^{\nu+1}+p x- q$ can be factored as $(x-1)(p+1+x+x^2+\ldots+x^{\nu})$, which also shows that $\bar{x}^{\nu+1}+p \bar{x}- q=0$ has a solution
$\bar{x}=1.$
Lastly, if $q>p+1$. 
Then, $\bar{x}\left(p+\bar{x}^{\nu}\right)>p+1$, showing that $\bar{x}>1$. 
This proves the theorem.
\end{proof}

\begin{theorem}\label{oscillation}
Let $q=p+1$ and $\{x_n\}_{n=0}^{\infty}$ be a positive solution of equation \eqref{problem1}, then $x_n$ oscillates at $\bar{x}=1$.
\end{theorem}
\begin{proof}
Let $q=p+1$ and $\{x_n\}_{n=0}^{\infty}$ be a positive solution of equation \eqref{problem1} then $\bar{x}=1$ is an equilibrium. Hence,
\[
x_{n+1} - \bar{x} =\frac{p+1}{p+x_n^{\nu}}- 1=\frac{1-x_n^{\nu}}{p+x_n^{\nu}}
\quad\Longleftrightarrow\quad
(x_{n+1}-1)(1-x_n^{\nu}) = \frac{(1-x_n^{\nu})^2}{p+x_n^{\nu}}>0.
\]
Suppose (WLOG) that $x_{n+1}-1>0$ and $1-x_n^{\nu}>0$. 
Then, $x_n<1<x_{n+1}$.
Now, 
\[
x_{n+1} - x_n =\frac{q}{p+x_n^{\nu}}- \frac{q}{p+x_{n-1}^{\nu}}
\ \ \Longleftrightarrow\ \
\frac{x_{n+1} - x_n}{x_{n-1}^{\nu} - x_n^{\nu}}= \frac{q}{(p+x_n^{\nu})(p+x_{n-1}^{\nu})}>0.
\]
Thus, $x_{n+1} > x_n$ and $x_{n-1} > x_n$, proving the theorem.
\end{proof}

Let $q=p+1$ and consider equation \eqref{problem1}. 
Linearizing \eqref{problem1} about the equilibrium point $\bar{x}=1$ we get $u_{n+1} + \nu u_n/(p+1)=0$. 
Therefore, its characteristic equation is $\lambda^n\left[\lambda+ \nu/(p+1)\right]=0$ whose roots are $\lambda=0,-\nu/(p+1)$.
With these results, we easily obtain the following theorem.

\begin{theorem}
Assume that $\nu < q=p+1$. 
Then, the unique positive equilibrium point $\bar{x}=1$ of \eqref{problem1} is locally asymptotically stable. 
\end{theorem}

\begin{theorem}
Assume that $\nu \geq q=p+1$. 
Then, \eqref{problem1} has a positive prime period two solution.
The prime period two solution of \eqref{problem1} takes the form
\[
\left\{\ldots, \frac{q}{p},\; \frac{q}{p+(q/p)^{\nu}},\; \frac{q}{p},\; \frac{q}{p+(q/p)^{\nu}}, \ldots\right\}.
\]
\end{theorem}

\begin{proof}
Let  $\ldots, \phi, \psi, \phi, \psi,\ldots $ be a period two solution of \eqref{problem1}. 
Then,
\begin{align}
\phi p +\phi \psi^{\nu} &= q\label{one},\\
\psi p + \phi^{\nu} \psi &= q\label{two}.
\end{align}
Subtracting \eqref{one} from \eqref{two}, we get
\[
(\psi - \phi)\left[p -\psi\phi\left(\frac{\psi^{\nu}-\phi^{\nu}}{\psi-\phi} \right)\right]=0.
\]
Since $\psi$ and $\phi$ are period two solutions, then $\phi\neq\psi$ and so, 
\begin{equation}\label{p}
p-\psi\phi\left(\frac{\psi^{\nu-1}-\phi^{\nu-1}}{\psi-\phi} \right)=0.
\end{equation}
Hence, we see that $\psi$ and $\phi$ are also solutions of \eqref{p}.
Now we multiply by $\psi$ and $\phi$ the equations \eqref{one} and \eqref{two}, respectively, and take the difference of the two resulting equations to obtain
\[
\psi\phi\left(\frac{\psi^{\nu}-\phi^{\nu}}{\psi-\phi} \right)=q.
\]
Thus, we obtain
\begin{equation}\label{qp}
\frac{\psi^{\nu}-\phi^{\nu}}{\psi^{\nu-1}-\phi^{\nu-1}}=\frac{q}{p}.
\end{equation}
The solution $\psi$ and $\phi$ to equation \eqref{qp} is the period two solution of \eqref{problem1} for $\nu>q=p+1$. 
Now, suppose that $\phi>\psi$. 
Following the proof of Theorem \ref{oscillation} we can show that $0<\psi<1<\phi$. 
Furthermore, it is true that $0<\psi^{\nu}\ll1$. 
Hence, from \eqref{one} and \eqref{two}, we have $\phi=q/p$ and $\psi=q/(p+(q/p)^{\nu})$, completing the proof of the theorem.
\end{proof}

We now turn our attention to the second equation with $\nu >1$.
\subsection{On equation $x_{n+1}=q/(-p+x_n^{\nu})$}

The following results can be verified easily.

\begin{theorem}
Every negative solution of \eqref{problem2} is bounded. 
\end{theorem}

\begin{theorem}
Let $\nu$ be a positive odd integer and $\bar{x}$ be an equilibrium point of equation \eqref{problem2}. Then, the following statements are true:
\begin{enumerate}
\item[\rm (i)] if $q<p+1$ then \eqref{problem2} has a unique negative equilibrium $\bar{x}$ inside the interval $(-1,0),$
\item[\rm (ii)] if $q=p+1$ then $\bar{x}=-1$ is a unique negative equilibrium of \eqref{problem2},
\item[\rm (iii)] if $q>p+1$ then \eqref{problem2} has a unique negative equilibrium $\bar{x} \in (-\infty,-1).$
\end{enumerate}
\end{theorem}

\begin{theorem}\label{equilibria}
Let $\nu$ be a positive even integer and $\bar{x}$ be an equilibrium point of equation \eqref{problem2}. Then, the following statements are true:
\begin{enumerate}
\item[\rm (i)] if $q<p-1$ then \eqref{problem2} has two negative equilibrium points, i.e., one equilibrium $\bar{x}$ in $(-1,0)$ and one inside the interval $(-\infty,-1)$,
\item[\rm (ii)] if $q=p-1$ then $\bar{x}=-1$ is a unique negative equilibrium of \eqref{problem2},
\item[\rm (iii)] if $q>p-1$ then \eqref{problem2} has no negative equilibrium point. 
\end{enumerate}
\end{theorem}
\begin{proof}

Let $\nu$ be an even integer and $\bar{x}$ be an equilibrium of \eqref{problem2}. 
Consider the function $G(x) = x^{\nu+1}-px-q$.
Then,
\begin{center}
$G(0)=-q<0, \quad G(-1)=p-1-q,$ \quad and \quad $\lim_{x\rightarrow-\infty} G(x)=-\infty.$  
\end{center}
If $q<p-1$ then $G(-1)>0$. 
This implies that $G(x)$ has two negative real roots, one in $(-1,0)$ and one in $(-\infty,-1)$. 
If $q=p-1$, then $G(x)$ has a unique negative real root $x=-1$. 
Finally, if $q>p-1$ then we obtain no negative real root for $G(x)$. 
Statements (i), (ii) and (iii) of Theorem \ref{equilibria} follows immediately.
\end{proof}

\begin{theorem}
Let $\nu$ be an odd natural number such that $\nu > q=p+1$ and $\{x_n\}_{n=0}^{\infty}$ be a negative solution of equation \eqref{problem2}. 
Then, $x_n$ oscillates at $\bar{x}=1$. 
\end{theorem}


Let $\bar{x}$ be an equilibrium of \eqref{problem2} and consider the function $H(x)=q/(-p+x^{\nu})$.
Since $H'(x)=-q\nu x^{\nu-1}/(-p+x^{\nu})^2$, then linearizing \eqref{problem2} about the equilibrium point $\bar{x}$, we get
$u_{n+1}-[q\nu \bar{x}^{\nu-1}/(-p+\bar{x}^{\nu})^2]u_n=0$.
Hence, its characteristic equation is given by 
$
\lambda^n\left(\lambda-[q\nu \bar{x}^{\nu-1}/(-p+\bar{x}^{\nu})^2]\right)=0,
$
whose roots are $\lambda=0$ and $\lambda = q\nu \bar{x}^{\nu-1}/(-p+\bar{x}^{\nu})^2$.
By Theorem \ref{theoremA}, equation \eqref{problem2} is stable provided 
\begin{equation}\label{ineq}
\left|\frac{q\nu \bar{x}^{\nu-1}}{(-p+\bar{x}^{\nu})^2}\right|<1.
\end{equation}
Suppose that $\nu$ is odd. 
Then, the equilibrium point $\bar{x}=-1$ when $q=p+1$ is stable for $\nu< p+1$ and unstable for $\nu \geq p+1$. 
This is also true for the equilibrium point $\bar{x}\in(-\infty,-1)$ when $q>p+1$. 
On the other hand, the equilibrium point $\bar{x}\in(-1,0)$ when $q<p+1$ is always stable for any odd number $\nu>0$.  
Now if $\nu$ is even, then the equilibrium point $\bar{x}=-1$ when $q=p-1$ is always stable for any even number $\nu>0$. 
This is also true for the equilibrium point $\bar{x}\in(-1,0)$ when $q<p-1$. 
With these results, we easily obtain the following theorems. 

\begin{theorem}
Let $\nu$ be an odd integer such that  $\nu< q=p+1$. 
Then, the unique negative equilibrium point $\bar{x}=-1$ of \eqref{problem2} is locally asymptotically stable. 
\end{theorem}

\begin{theorem}
Assume that $q=p-1$. 
Then, the unique negative equilibrium point $\bar{x}=-1$ of \eqref{problem2} is locally stable. 
\end{theorem}

\begin{theorem}
Let $\nu$ be an odd natural number such that $\nu\geq q=p+1$. 
Then, equation \eqref{problem2} has a prime period two solution. The prime period two solution takes the form:
\[
\left\{\ldots, -\frac{q}{p},\; -\frac{q}{p+(q/p)^{\nu}},\; -\frac{q}{p},\; -\frac{q}{p+(q/p)^{\nu}}, \ldots\right\}.
\]
\end{theorem}

\begin{theorem}
Let $\nu$ be an even natural number. 
Then, there exists some natural number $N>q=p+1$ such that for every $\nu\geq N,$ equation \eqref{problem2} has a prime period two solution. 
The prime period two solution takes the form:
\[
\left\{\ldots, -\frac{q}{p},\; \frac{q}{-p+(q/p)^{\nu}},\; -\frac{q}{p},\; \frac{q}{-p+(q/p)^{\nu}}, \ldots\right\}.
\]
\end{theorem} 
In what follows, we give some numerical examples to illustrate our previous results. 
Figure \ref{fig6} illustrate our results for the nonlinear difference equation \eqref{problem1} with $\nu>1$. 
Meanwhile, Figure \ref{fig7} shows the behavior of solutions for the nonlinear difference equation \eqref{problem2} with $\nu$ at least $2$. 
The values for $p,q, \nu$ and the initial condition $x_0$ are indicated in each of the given plot. 

\begin{figure}[h!]
   \centering
    \scalebox{0.65}{\includegraphics{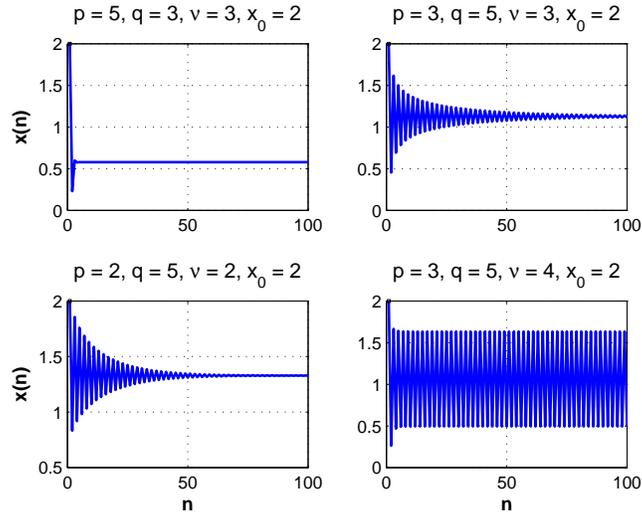}}
    \caption{The above plots illustrate the long term dynamics of the nonlinear difference equation $x_{n+1} = q/(p+x_n^{\nu})$ for some values of the parameter set $(p,q,\nu,x_0)$.}
    \label{fig6}
\end{figure}

\begin{figure}[h!]
   \centering
    \scalebox{0.65}{\includegraphics{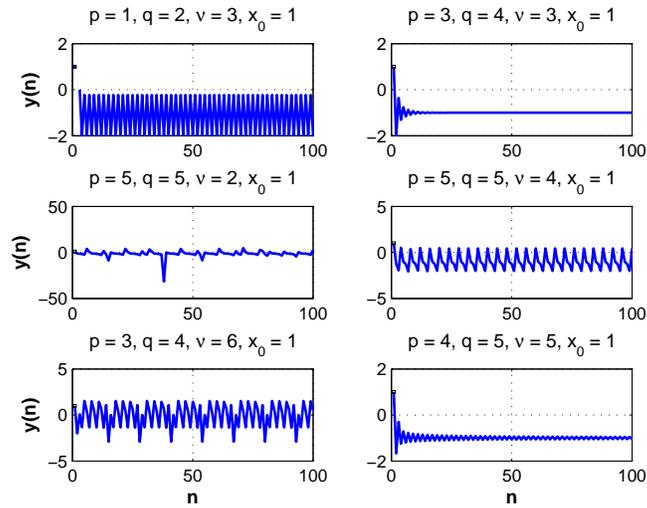}}
    \caption{The above plots illustrate the long term dynamics of the nonlinear difference equation $y_{n+1} = q/(-p+y_n^{\nu})$ for some values of the parameter set $(p,q,\nu,y_0)$.}
    \label{fig7}
\end{figure}

{\bf Authors' Note} After the first version of this paper has been drafted (March 23, 2014), 
we have learned that the solution form of the Riccati difference equation has been solved completely in 
[Representation of solutions of bilinear equations in terms of generalized Fibonacci sequences, {\it Electron. J. Qual. Theory Differ. Equ.}, {\bf 67} (2014) 1--15] by St\'{e}vic. 
However, as alluded in the introduction, the solution form of the Riccati difference equation was first obtained by Brand in [A sequence defined by a difference equation, {\it Am. Math. Mon.}, {\bf 62} (1955), 489--492] but this work was not mentioned by St\'{e}vic in his paper. Nevertheless, the results presented here, except possibly for the form of solution of the two nonlinear difference equations \eqref{problem1} and \eqref{problem2}, are new and are of different interest from those we have mentioned.

\section{Summary and Future Work}

In this work we have investigated the behavior of solutions of two special types of Riccati difference equation of the form $x_{n+1} = q/(\pm p + x_n)$.
It was shown that the solution of these equations are expressible in terms of the well-known Horadam sequence.
Two similar equations of the form $x_{n+1} = q/(\pm p + x_n^{\nu})$, where $\nu > 1$, were also examined.
Apparently, the stability of the equilibrium points of these equations behave differently according to some conditions imposed on the parameters $p$, $q$ and $\nu$.
As verified through numerical experiments, the difference equation $x_{n+1} = q/(\pm p + x_n^{\nu})$ may have a prime period two solution whenever the inequality condition $\nu \geq q = p+1$ is satisfied.  
In our next investigation, we shall study the dynamics of the coupled difference equation given by the system
\[
x_{n+1} = \frac{a}{b+ y_n^{\mu}}, \qquad y_{n+1} = \frac{\alpha}{\beta+ x_n^{\nu}}, \qquad n =0,1, \ldots, 
\]
where $a, b, \alpha$ and $\beta$ are real numbers and $x_0$ and $y_0$ are real positive values. 

%

\footnotesize


\noindent
email:\ journal@monotone.uwaterloo.ca\\
http://monotone.uwaterloo.ca/$\sim$journal/

\end{document}